\documentclass[12pt]{amsart}
\usepackage[utf8]{inputenc}
\usepackage{amsthm}
\usepackage{amsmath}
\usepackage{amssymb}
\usepackage{comment}
\usepackage{tikz}
\usepackage[margin=1.2in]{geometry}

\usepackage[T1]{fontenc}
\usepackage{enumitem}
\usepackage{setspace}
\usepackage{microtype}
\usepackage{cite}

\allowdisplaybreaks

\usepackage[colorlinks=true, pdfstartview=FitV, linkcolor=blue, citecolor=blue, urlcolor=blue]{hyperref}

\newtheorem{theorem}{Theorem}[section]

\newtheorem{lemma}[theorem]{Lemma}

\newtheorem{corollary}[theorem]{Corollary}

\newtheorem{conj}[theorem]{Conjecture}

\providecommand{\ol}{\overline}

\providecommand{\wh}{\widehat}

\providecommand{\C}{\mathbb C}

\providecommand{\R}{\mathbb R}

\providecommand{\T}{\mathcal{T}}

\providecommand{\M}{\mathcal{M}}

\title[Small-Cap Decoupling in $\R^4$]{A continuum of Small-cap decouplings and Exponential Sums for the Moment Curve in $\R^4$}

\author{Jacob Glidewell}
\address{Dept. of Mathematics \\
Indiana University \\
 Bloomington, IN, 47405-7106, USA}
\email{jaglide@iu.edu}

\begin{document}

\begin{abstract}
    We use the high-low method and wavepacket pruning to prove new small-cap decoupling estimates for the moment curve in $\R^4$. As an application, we verify a conjecture of Demeter regarding the $L^{12}$ square-root cancellation of exponential sums associated with the moment curve in $\R^4$. This provides a continuum of square-root cancellation estimates that connects the Vinogradov MVT in $\R^3$ with a result of Bourgain, related to improving the best known estimate for the Lindel\"{o}f hypothesis. 
\end{abstract}

\maketitle

\section{Introduction}

We are motivated to attack the following conjecture of Demeter \cite[Conjecture 8.1]{demeter2021l12squarerootcancellation} for the case of the moment curve $\Phi(t)= (t, t^2, t^3, t^4)$. The $p=12$ case is given in \cite[Conjecture 1.1]{demeter2021l12squarerootcancellation}. We use the notation $e(x) = e^{2\pi i x}$ for $x\in \R$.
\begin{conj}\label{conj:L12}
    Let $11\le p\le 12$. Assume $a\ge b\ge 0$ and $a+b=\frac{p}{2}-3$. Let $\Omega = [0,N]\times [0,N^2]\times [0,N^a]\times [0,N^b]$. Then, \begin{equation}\label{eq:L12expsumconj}
        \int_\Omega \left|\sum_{n\sim N}e(x\cdot \Phi\left(\frac{n}{N}\right))\right|^{p} dx \lessapprox N^{\frac{p}{2}}|\Omega|.
    \end{equation}
\end{conj}
Some special cases of this are known. Notably, the case $(p, a,b)=(12, 3,0)$ corresponds to the Vinogradov Mean Value Theorem in $\R^3$ (\cite{Wooley_2016}, \cite{bourgain2016proofmainconjecturevinogradovs}). Indeed, using periodicity and VMVT, we get the estimate for each $x_4\in [0,1]$, $$\frac{1}{N^3}\int_{[0,N^3]^3}\left|\sum_{n\sim N}e(x\cdot \Phi\left(\frac{n}{N}\right)) \right|^{12}\; dx_1dx_2dx_3 \lessapprox N^{6}\cdot N^6.$$ Integrating in $x_4$ gives \eqref{eq:L12expsumconj} in this case. The case $(p, a,b)=(12, 2,1)$ was shown by Bourgain \cite{bourgain2016decouplingexponentialsumsriemann}. In that paper, the following variant of \eqref{eq:L12expsumconj} served as the main ingredient to his improvement on the Lindel\"{o}f hypothesis: \begin{equation}\label{eq:BL12}
    \int_{\Omega}\left|\sum_{n\sim N}e(x\cdot \tilde{\Phi}\left(\frac{n}{N}\right))\right|^{12} dx \lessapprox N^{6}|\Omega|
\end{equation} with $\tilde{\Phi}(t) = (t,t^2, t^{3/2}, t^{1/2})$. In particular, Bourgain showed \eqref{eq:BL12} for any curve $\tilde{\Phi}(t)=(t, t^2, \phi_3(t), \phi_4(t))$ satisfying some nondegeneracy and torsion conditions. In \cite{demeter2021l12squarerootcancellation}, Demeter conjectured a continuum of estimates connecting these two endpoints, and proved the related cases where $p=12$ and $a\in [\frac{3}{2}, 2]$. All of these results are part of a vast program in Fourier analysis, initiated in \cite{demeter2020smallcapdecouplings}, called small-cap decouplings which have become useful for proving exponential sum estimates on spatial truncations of the torus. To the author's knowledge, no prior estimates towards Conjecture \ref{conj:L12} are known for the case when $p<12$ and $b>0$.

Bourgain's original approach in \cite{bourgain2016decouplingexponentialsumsriemann} relies on a geometric observation that a bilinear version of \eqref{eq:BL12} reduces matters to quadratic exponential sums in the plane. To show this bilinear version, he iterates $L^6$ decouplings with a "magic" change of variables. One key feature of his argument uses that the domain $[0,N]$ for $x_4$ matches the period of $x_1$. His change of variables miraculously exploits this fact to move "volume" from $x_4$ to $x_1$. See \cite[Section 3]{demeter2021l12squarerootcancellation} for a discussion of this argument.
The method in \cite{demeter2021l12squarerootcancellation} generalizes Bourgain's approach. By combining $L^6$ decoupling for the parabola with Gauss sum estimates, Demeter was able to prove the sharp estimates in the regime $a\in [\frac{3}{2}, 2]$. However, in both versions of the argument, for the cases $a> 2$, this mechanism becomes inefficient due to the loss of genuine 4-dimensional behavior at the global scale. A direct manifestation of this is that the domain in $x_4$ is now shorter than the period of $x_1$.
Rather than clever manipulations of existing decouplings, we will progress Conjecture \ref{conj:L12} by new small-cap decouplings. The new technology we use for this problem is the small-cap decoupling for the cone, developed by Guth and Maldague  \cite{maldague2022amplitudedependentwaveenvelope}, \cite{Guth_2024}.

Let $\beta \in [\frac{1}{3}, \frac{1}{2}]$ and $\alpha\in (0, \beta]$. Define \begin{multline*}
    \M^4(R^{\beta}, R) = \big\{(\xi_1, \xi_2, \xi_3, \xi_4) \; : \xi_1 \in [\frac{1}{2}, 1], |\xi_2-\xi_1^2|\le R^{-2\beta},\\ |\xi_3-3\xi_1\xi_2+2\xi_1^3|\le R^{-1}, |\xi_4+4\xi_1\xi_3-9\xi_2^2+4\xi_1^4|\le R^{-1}\big\}.
\end{multline*} We define the small-caps $\gamma$ as the partition $\M^4(R^{\beta}, R)\cap \{\ell R^{-\beta}\le \xi_1\le (\ell+1)R^{-\beta}\}$ for each integer $\ell\in R^{\beta}[\frac{1}{2}, 1]$. We will show the following generalized square-root cancellation phenomenon. Given $F:\R^4\to \C$, we define $F_\gamma$ as the Fourier projection onto $\gamma$. For clarity, we call a $S_1\times S_2\times S_3\times S_4$ box in $\R^4$ \emph{regular} if the side with length $S_i$ is parallel to the $x_i$-axis. 
\begin{theorem} \label{thm:mainResult}
   For any $F:\R^4\to \C$ Schwartz with Fourier Transform supported in $\M^4(R^{\beta}, R)= \bigcup_{\gamma}\gamma$ and any $R\times R\times R\times R^{\alpha}$ regular box $\Omega_R$ , if $\|F_\gamma\|_\infty \le 1$, we have $$\|F\|_{L^{p}(\Omega_R)}^{p}\lessapprox R^{\frac{p}{2}\beta} |\Omega_R|,$$ with $$p= 6+\frac{2(\alpha+1)}{\beta},$$ for $\alpha$ and $\beta$ in the range corresponding to $p\ge 11$.
\end{theorem}

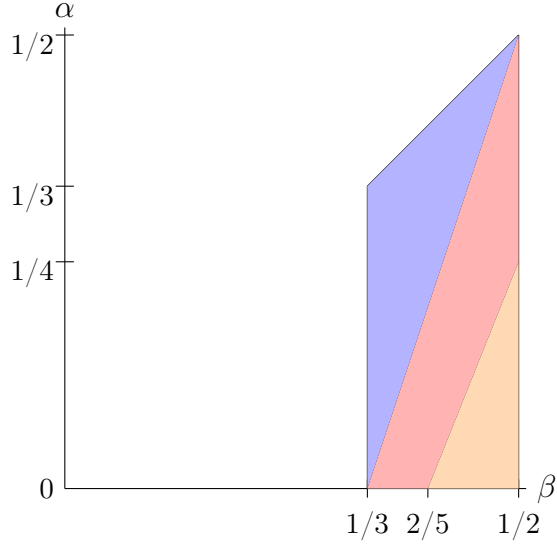
\begin{figure}[h] \label{fig:trapezoid}
    \centering
      \begin{tikzpicture}[domain=0:6]

  \draw[-] (0,0) -- (6.1,0) node[right] {$\beta$};
  \draw[-] (0,0) -- (0,6.1) node[above] {$\alpha$};

\draw (6,0.12) -- (6,-0.12);
			\node[below] at (6,-0.15) {\small 1/2};

\draw (4.8,0.12) -- (4.8,-0.12);
			\node[below] at (4.8,-0.15) {\small 2/5};

\draw (4,0.12) -- (4,-0.12);
			\node[below] at (4,-0.15) {\small 1/3};

\draw (0.12, 4) -- (-0.12, 4);
\node[left] at (0, 3.85) {\small 1/3};
\draw (0.12, 6) -- (-0.12, 6);
\node[left] at (0, 5.85) {\small 1/2};
\draw (0.12, 3) -- (-0.12, 3);
\node[left] at (0, 2.85) {\small 1/4};

\node[left] at (0, 0) {\small 0};

\draw (4, 0) -- (4, 4);

\draw (4, 4) -- (6, 6);

\draw (6,6) -- (6,0);

\draw (6,6) -- (4,0);

\draw (6,3) -- (4.8,0);

\fill[color=blue!30] (4,0) -- (4,4) -- (6,6) -- cycle;
\fill[color=red!30] (4,0) -- (6,6) -- (6,3) -- (4.8,0) -- cycle;
\fill[color=orange!30] (4.8,0) -- (6,3) -- (6,0) -- cycle;
\end{tikzpicture}
    \caption{Parameters $\alpha$ and $\beta$. In the blue region, the critical exponent is $p\in(12, 14)$. In the red, $p\in (11, 12)$. In the orange, $p<11$.}
    
\end{figure}

We call this phenomenon "square-root cancellation" because, on average, $F$ is comparable to  $R^{\beta/2}$. This is seen easily with the $L^2$-moment: $\|F\|_{L^2(\Omega_R)}\lesssim R^{\beta/2}|\Omega_R|^{1/2}$. We compare this to the triangle inequality, which gives $$|F(x)|\le \sum_{\gamma}|F_\gamma(x)|\lesssim R^{\beta}.$$ When $|F(x)|\sim R^{\beta}$, we call this "constructive interference." It is easy to see that for exponential sums, such as \eqref{eq:L12expsumconj} or in Corollary \ref{cor:NewExpSums}, constructive interference occurs in an $O(1)$-neighborhood of $x=0$. The critical exponent $p$ is the $L^p$-moment where the competing behavior of square-root cancellation and constructive interference match. For larger moments, constructive interference dominates while for smaller moments, square-root cancellation dominates. Observe that by H\"{o}lder's inequality, Theorem \ref{thm:mainResult} is true more generally for $2\le p\le 6+\frac{2(\alpha+1)}{\beta},$ so for the rest of the paper we will only consider $$p= 6+\frac{2(\alpha+1)}{\beta}.$$ Notice $p\in (10, 14]$ for the range of parameters; see Figure \ref{fig:trapezoid}. When $10< p< 11$, we expect an analogous $4$-linear estimate; however, our methods only show a $4$-broad version. A counterexample for the linear estimate when $p<11$ can be found in \cite[Theorem 8.2]{demeter2021l12squarerootcancellation}.
\begin{theorem}\label{thm:p<11}
    For any $F:\R^4\to \C$ Schwartz with Fourier Transform supported in $\M^4(R^{\beta}, R)= \bigcup_{\gamma}\gamma$, any $R\times R\times R\times R^{\alpha}$ regular box $\Omega_R$, and any caps $\tau_1, \tau_2, \tau_3, \tau_4$ with $|\tau_i|\sim 1$ and $\sim 1$-separated, if $\|F_\gamma\|_\infty \le 1$, we have $$\|\min_{1\le i\le 4}|F_{\tau_i}|\|_{L^{p}(\Omega_R)}^{p}\lessapprox R^{\frac{p}{2}\beta} |\Omega_R|,$$ with $$p= 6+\frac{2(\alpha+1)}{\beta},$$ for the parameters $\beta\in [\frac{1}{3}, \frac{1}{2}]$ and $0<\alpha\le \beta$.
\end{theorem}

As a corollary of Theorem \ref{thm:mainResult}, we verify  Conjecture \ref{conj:L12} in the remaining range for $p=12$, make progress for $p\in [11, 12)$, and give new exponential sum estimates for $p\in (12, 14]$. These estimates with $c_n=1$ and constructive interference on $[0, \frac{1}{100}]^4$ show that the critical exponent $p$ in Theorem \ref{thm:mainResult} is sharp.

\begin{corollary}\label{cor:NewExpSums} Let $11\le p\le 14$. Assume $a\in [2,3]$, $b\in (0,1]$ and $a+b=\frac{p}{2}-3$.  Let $\Omega = [0,N]\times [0,N^2]\times [0,N^a]\times [0,N^b]$. Then, for any $|c_n|\le 1$,  $$\int_\Omega \left|\sum_{n\sim N}c_ne(x\cdot \Phi\left(\frac{n}{N}\right))\right|^{p} dx \lessapprox N^{\frac{p}{2}}|\Omega|.$$
\end{corollary}
Therefore, the remaining open range in Conjecture \ref{conj:L12} is $p\in [11, 12)$ with $\frac{p-6}{4}\le a<2$. This is analogous to the range for $p=12$ which Demeter proved in \cite{demeter2021l12squarerootcancellation}, which our methods seem to not be able to address. These correspond to decouplings with parameters $\alpha>\beta$ in the language of Theorem \ref{thm:mainResult}.

To establish Theorem \ref{thm:p<11}, we use a high-low stopping time argument with wavepacket pruning, combining the approaches of Guth and Maldague \cite{Guth_2024} and Maldague and Oh \cite{maldague2024smallcapdecouplingmoment}. By a standard broad-narrow argument, we get the linear inequalities in Theorem \ref{thm:mainResult}. The main novelity of this paper is the integration and simplification of the arguments from \cite{Guth_2024} and \cite{maldague2024smallcapdecouplingmoment}. Further, we elaborate on the lack of linear inequalities for $p<11$. One obstruction that we see explicitly is that the rescaling needed for the reduction is inefficient when $p<11$.

\medskip
\noindent
{\it Acknowledgments}. The author would like to thank his advisor, Ciprian Demeter, for introducing him to this problem and for encouragement throughout the completion of this project.

\section{Notation and Outline}\label{sec:Outline}
We write $A\lesssim B$ or $A= O(B)$ if $A\le cB$ for some universal constant (possibly large) $c>0$. Also, we write $A\sim B$ if $A\lesssim B$ and $B\lesssim A$. We write $A\lessapprox B$ if $A\lesssim_\epsilon R^{\epsilon} B$ for any $\epsilon>0$. Finally, we write $A\approx B$ if $B\lesssim A\lessapprox B $.

Now, we define the operational notation for our argument. Fix $\epsilon>0$. Define intermediate scales $R_k = R^{k\epsilon}$ for $k=0, \dots, K$ where $R_K\sim R^{\beta}$ with $K\sim \frac{1}{\epsilon}$. Now, we define the intermediate caps. \begin{enumerate}
    \item[i)] For $R_k\le R^{1/4}$, let $\gamma_k$ be the caps with dimensions $$\sim R_k^{-1}\times R_k^{-2}\times R_k^{-3}\times R_k^{-4}.$$

    \item[ii)] For $R^{1/4}<R_k\le R^{1/3}$, let $\gamma_k$ be the caps with dimensions $$\sim R_k^{-1}\times R_k^{-2}\times R_k^{-3}\times R^{-1}.$$

    \item[iii)] Finally, for $R^{1/3}<R_k\le R^{\beta}$, let $\gamma_k$ be the caps with dimensions $$\sim R_k^{-1}\times R_k^{-2}\times R^{-1}\times R^{-1}.$$
\end{enumerate} 
These coincide with the cap sizes in \cite{Guth_2024} in the first three dimensions. Fix $A>0$ a constant to be determined, depending on $\epsilon$. There are two places in the argument (namely, Lemma \ref{lem:PassToPruned} and estimating the low piece in Section \ref{sec: 8+4=12}) where the size of $A$ is determined. Fix $\tau_1, \tau_2, \tau_3, \tau_4$ caps with $|\tau_i|\sim 1$ and $1$-separated. For $\lambda>0$, let $$U_{\lambda} = \{x\in \Omega_R : \lambda\le\min_{1\le i\le 4}|F_{\tau_i}(x)|\le 2\lambda \}.$$ We normalize so that $\max_\gamma\|F_\gamma\|_\infty =1$. By dyadic pigeonholing, we have for some $0<\lambda\le R^{\beta}$ ,
\begin{equation}\label{eq:4-linDyadicPig}
\|\min_{1\le i\le 4}|F_{\tau_i}|\|_{L^{p}(\Omega_R)}^{p}\lessapprox \lambda^p|U_{\lambda}|.
\end{equation} We assume without loss of generality that $\lambda\ge R^{\beta/2}$. Otherwise, Theorem \ref{thm:p<11} is immediate. We will make use of this assumption in Section \ref{sec:proof}.

At this point, let us give a heuristic overview of the proof of Theorem \ref{thm:p<11}; the broad-to-linear reduction for Theorem \ref{thm:mainResult} is standard and is given in Section \ref{sec:multRed}. The argument for Theorem \ref{thm:p<11} nicely decomposes into two major phases: (a) finding local square-root cancellation and (b) gaining globally with a high-low decomposition. \begin{enumerate}
    \item[Phase (a):] Our initial goal is to identify where in $\Omega_R$ square-root cancellation is occurring locally. This is quantified by square functions at each scale $R_k$. Indeed, we would like $F$ to be comparable to $$\left(\sum_{\gamma}|F_\gamma|^2\right)^{1/2}\lesssim R^{\frac{\beta}{2}}.$$ However, this is not always true. Instead, we pick out the spatial places where intermediate (squared) square functions are comparable to square-root cancellation; that is, \begin{equation}\label{eq:sqrtCancellation}
        g_k=\sum_{\gamma_k} |F_{\gamma_k}|^2 \sim R^{\beta}.
    \end{equation} Each of these square functions are essentially constant on balls of radius $R_k$, so we decompose via a stopping time selection $\Omega_R=\bigcup \Lambda_k$ with each $\Lambda_k$ a union of $R_k$-cubes where \eqref{eq:sqrtCancellation} occurs. To pass from $F$ to these intermediate square functions, we use the Multilinear Restriction Theorem for the moment curve (see \cite[Ex. 6.33]{demeter2020fourier} or Theorem \ref{thm:MultRestrictionMomentCurve} in the next section). This preserves the square-root cancellation from $F$ to the intermediate square function on $\Lambda_k$, giving the estimate which "morally" says $$\lambda^8|U_{\lambda}\cap \Lambda_k|\lessapprox R^{4\beta}|\Lambda_k|.$$

    \item[Phase (b):] Now, it suffices to get a gain on the size of $\Lambda_k$. Because of the stopping time selection and our expectation that $F$ satisfies square-root cancellation on most of $\Omega_R$, we expect $|\Lambda_k|$ to decrease as $R_k$ increases. For instance, when $p=12$, we would like to show $$|\Lambda_k|\lessapprox \frac{R^{2\beta}}{\lambda^4}|\Omega_R|.$$ As we will see in Lemma \ref{lem:upgrader}, the critical $\lambda$ to consider is $\lambda\sim R_k^{1/2}R^{\frac{\beta}{2}}.$ Thus, the estimate we would like to show is "morally" $$|\Lambda_k|\lessapprox \frac{1}{R_k^2}|\Omega_R|,$$ which matches our expectation. To prove such an estimate, we convert to estimating a $L^q$-moment of the intermediate (squared) square function: $$R^{q\beta}|\Lambda_k|\sim \int_{\Lambda_k}\left(\sum_{\gamma_k}|F_{\gamma_k}|^2\right)^{q} = \int_{\Lambda_k}|g_k|^q.$$ Because of the stopping time selection of $\Lambda_k$, we have control over the next intermediate square function $$g_{k+1}=\sum_{\gamma_{k+1}}|F_{\gamma_{k+1}}|^2.$$ We relate the two square functions by the following identity:

    \begin{align*}
g_k=\sum_{\gamma_k}|F_{\gamma_k}|^2&= \sum_{\gamma_k}\sum_{\gamma_{k+1}, \gamma_{k+1}^\prime\subset \gamma_k}F_{\gamma_{k+1}}\ol{F_{\gamma_{k+1}^\prime}}\\
        &= \sum_{\gamma_k}\sum_{\gamma_{k+1}\sim \gamma_{k+1}^\prime\subset \gamma_k}F_{\gamma_{k+1}}\ol{F_{\gamma_{k+1}^\prime}}+\sum_{\gamma_k}\sum_{\gamma_{k+1}\not\sim \gamma_{k+1}^\prime\subset \gamma_k}F_{\gamma_{k+1}}\ol{F_{\gamma_{k+1}^\prime}},
    \end{align*} where $\sim$ is the adjacency relation (i.e $\gamma_{k+1}\sim \gamma_{k+1}^\prime$ if they are the same or adjacent). We can bound the first term by $g_{k+1}$; this is called the "low" piece. It suffices to control the second term; this is called the "high" piece. We will explain in Section \ref{sec: 8+4=12} that the terms in the high piece exhibit additional orthogonality, due to the fact that their frequency supports mimic that of cone planks. This is the main mechanism of our proof, mirroring that of \cite{Guth_2024}. Since we are proving estimates over a range of parameters, some combinations of $\alpha$ and $\beta$ have access to better or worse outcomes from the Multilinear Restriction Theorem (Part (a)) and cone decouplings (Part (b)). From this obstacle, we finally need to show how these estimates combine to give the correct estimates for Theorem \ref{thm:p<11}. We must consider three different regimes of $R_k$ where slightly different arguments are needed.
\end{enumerate}
A key technical tool used in the argument is "wavepacket pruning" (see Lemma \ref{lem:PruningProperties}). This tool is commonly used in tandem with the high-low stopping time argument. It "prunes off" irrelevant wavepackets with large amplitudes relative to each intermediate square function. One should observe that wavepacket pruning is not strictly needed for every iteration of this kind of argument. For instance, \cite{maldague2024smallcapdecouplingmoment} does not make use of any pruning. For similar reasons, Theorem \ref{thm:mainResult} with $\beta = \frac{1}{3}$ does not need wavepacket pruning. However, the other cases $\beta>\frac{1}{3}$ need this tool and so we use it for all cases. One should note that pruning causes the need for a $4$-broad formulation, rather than a $4$-linear formulation; see Lemma \ref{lem:upgrader}.

A striking similarity between our argument here and Bourgain's original argument is that the $x_4$ variable becomes irrelevant after an initial "local" step. In Bourgain's argument, the local step was an initial application of the $L^6$ decoupling inequality on $R^{1/2}$-cubes while in our argument, this was part (a) described above. It remains unclear exactly why (other than for ad-hoc reasons) one may ignore $x_4$ in some steps and still produce sharp estimates. One may hope that a general theory of these decouplings may shed some light on this dependence. 

The order of this paper is as follows. Section \ref{sec:oldDecouplings} gives a review of the known decouplings we will be using for our argument. Sections \ref{sec:Pruning} and \ref{sec:stopTime} present the pruning and stopping time setup for the argument, respectively. In section \ref{sec:multRestriction}, we will perform our initial $L^8$ estimate using multilinear restriction estimates for the moment curve. This exploits the 4-dimensional setting of the problem. Sections \ref{sec: 8+4=12} and \ref{sec:proof} will present how we use decoupling for the cone in $\R^3$ to enhance our estimate from $L^8$ to $L^p$. This enhancement mechanism ignores $x_4$ in our argument, so it is likely that other estimates beyond $L^{14}$ or in the remaining open range of Conjecture \ref{conj:L12} may be progressed by utilizing $x_4$ via new cone-like geometry in $\R^4$. Therefore, we give the heuristic that the estimates of Corollary \ref{cor:NewExpSums} are locally 4-dimensional but globally 3-dimensional. Finally, section \ref{sec:multRed} presents the broad-to-linear reduction and the necessity of $p\ge 11$. Section \ref{sec:FinalCor} gives the short proof of Corollary \ref{cor:NewExpSums} using Theorem \ref{thm:mainResult}.

\section{Known Restriction and Decoupling Estimates}\label{sec:oldDecouplings}

Let us recall some known restriction estimates and decouplings that we will use for our argument; namely, the mulitlinear restriction estimate for the moment curve \cite[Ex. 6.33]{demeter2020fourier}, the Bourgain-Demeter-Guth decoupling for the moment curve \cite{bourgain2016proofmainconjecturevinogradovs}, the Bourgain-Demeter canonical decoupling for the cone \cite{demeterCanonicalDecoupling2015}, and the Guth-Maldague small-cap decoupling for the cone \cite{maldague2022amplitudedependentwaveenvelope}.

Define $\M^n(R)$ to be the $\frac{1}{R}$-neighborhood of the moment curve in $\R^n$: $$r(t)=(t, t^2, t^3, \ldots, t^n), \; t\in [\frac{1}{2}, 1].$$ We define the canonical caps $\gamma$ as the partition $\M^n(R)\cap \{\ell R^{-1/n}\le \xi_1\le (\ell+1)R^{-1/n}\}$ for each integer $\ell\in R^{1/n}[\frac{1}{2}, 1]$. Let $\{\tilde{\tau_{i}}\}_{1\le i\le n}$ be caps with $|\tilde{\tau_i}|\sim 1$ and $\sim 1$-separated. 

\begin{theorem}\label{thm:MultRestrictionMomentCurve} \cite[Ex. 6.33]{demeter2020fourier}
    For any $F:\R^n\to \C$ with Fourier Transform supported in $\M^n(R)$ and any ball $B_R$ of radius $R$, we have $$\frac{1}{|B_R|}\int_{B_R}\left|\prod_{i=1}^nF_{\tilde{\tau_i}}\right|^{2}\lesssim \prod_{i=1}^n \left(\frac{1}{|B_R|}\int_{B_R}|F_{\tilde{\tau_i}}|^2\right).$$
\end{theorem}

We will make use of Theorem \ref{thm:MultRestrictionMomentCurve} for $\R^3$ and $\R^4$ in Section \ref{sec:multRestriction}.

\begin{theorem}\label{thm:BDGmomentcurve}
    \cite[Theorem 1.2]{bourgain2016proofmainconjecturevinogradovs} For any $F:\R^n\to \C$ with Fourier Transform supported in $\M^n(R) = \bigcup_{\gamma} \gamma$ and any ball $B_R$ of radius $R$, we have $$\int_{B_R}|F|^{n(n+1)} \lessapprox \left(\sum_{\gamma}\|F_\gamma\|_{L^{n(n+1)}(B_R)}^2\right)^{\frac{n(n+1)}{2}}.$$
\end{theorem} For $\R^2$, this is an $L^6$-decoupling and for $\R^3$, this is an $L^{12}$-decoupling. We will also frequently use the idea of cylindrical decoupling. Let us state this principle for $\R^4$. Let $F:\R^4\to \C$ with Fourier Transform supported in $\M^3(R)\times \R$. Let $B$ be a regular box which has lengths $R\times R\times R\times r$ for some $r\ge 1$. Then, $$\int_{B}|F|^{12}\lessapprox \left(\sum_{\gamma^*=\gamma\times \R}\|F_{\gamma^*}\|_{L^{12}(B)}^2\right)^{6}.$$ This follows by applying Theorem \ref{thm:BDGmomentcurve} on $F(\cdot, x_4)$ for each $x_4$ followed by integrating in $x_4$ and Minkowski's inequality; see \cite[Ex. 9.22]{demeter2020fourier}. The similar $L^6$ inequality holds with $\M^2(R)\times \R^2$ and regular boxes $B$ with lengths $R\times R\times r_1 \times r_2$.

Now, let us set up the cone. Let $\kappa \in [\frac{1}{2}, 1]$. Define the cone in $\R^3$ $$\Gamma = \{(u, uv, uv^2): u\in [\frac{1}{2}, 1], v\in [-1, 1]\},$$ and $\Gamma(R)$ the $\frac{1}{R}$-neighborhood of $\Gamma$. We define the small-planks $\theta_{\kappa}$ as the partition $\Gamma(R) \cap \{\ell R^{-\kappa}\le v \le (\ell+1)R^{-\kappa}\}$ for each integer $\ell \in R^\kappa[-1,1]$. We write $\theta= \theta_{1/2}$.

\begin{theorem}\label{thm:BDconedec}
    \cite[Theorem 1.2]{demeterCanonicalDecoupling2015}
For any $F:\R^3\to \C$ with Fourier Transform supported in $\Gamma(R) = \bigcup_{\theta} \theta$ and any ball $B_R$ of radius $R$, we have $$\int_{B_R}|F|^6 \lessapprox \left(\sum_{\theta}\|F_\theta\|_{L^6(B_R)}^2\right)^{3}.$$
\end{theorem}

\begin{theorem}\label{thm:small-capCone}\cite[Theorem 3]{maldague2022amplitudedependentwaveenvelope}  For any $F:\R^3\to \C$ with Fourier Transform supported in $\Gamma(R) = \bigcup_{\theta_\kappa} \theta_{\kappa}$ and any ball $B_R$ of radius $R$, we have $$\int_{B_R}|F|^q \lessapprox R\sum_{\theta_\kappa}\|F_{\theta_\kappa}\|_{L^q(B_R)}^q,$$ with $q = 2+\frac{2}{\kappa}.$ 
\end{theorem}
These two decouplings for the cone will be used in Section \ref{sec: 8+4=12} to produce our desired gain for $|\Lambda_k|.$

\section{Wavepacket Pruning}\label{sec:Pruning}
We first prune wavepackets to create the functions $F^{(k)}=\sum_{\gamma_k}F_{\gamma_k}^{(k)}$, identically as in \cite[Lemma 1]{Guth_2024}. We record the key properties here. The main point is $F^{(k)}$ has a similar Fourier support to $F$ with a much better $L^{\infty}$ size, suitable for transferring to smaller exponents where lower dimensional decoupling may be used.  

\begin{lemma}\label{lem:PruningProperties}
For each $\gamma_k$,
    \begin{enumerate}
        \item[a)] 
        $$|F^{(k)}_{\gamma_k}(x)|\lesssim |F^{(k+1)}_{\gamma_k}(x)|\le \sum_{\gamma_{k+1}\subset \gamma_k}|F_{\gamma_{k+1}}^{(k+1)}(x)|,$$

        \item[b)]  $$\|F_{\gamma_k}^{(k)}\|_\infty \lesssim \min \left\{\frac{R^{\beta}}{\lambda}, \frac{R^\beta}{R_k}\right\},$$

        \item[c)] $$\text{supp }\wh{F_{\gamma_k}^{(k+1)}}\subset \text{supp }\wh{F_{\gamma_k}^{(k)}}\subset C \gamma_k.$$
    \end{enumerate}
\end{lemma}

\begin{proof}
    We perform an iterative scheme. Define $F_{\gamma_{K}}^{(K)}= F_\gamma$ and $F_{\gamma_{K-1}}^{(K)}= F_{\gamma_{K-1}}$. For each $1\le k \le K-1$, write the wavepacket decomposition dual to $\gamma_{k}$: $$F^{(k+1)} = \sum_{\gamma_k} \sum_{T_{\gamma_k}}\psi_{T_{\gamma_k}}F_{\gamma_k}^{(k+1)}.$$ We split the $T_{\gamma_k}$ into two collections: good wavepackets and bad wavepackets. The good ones are defined by $$\T_{\gamma_k}^g = \{T_{\gamma_k}: \|\psi_{T_{\gamma_k}}F_{\gamma_k}^{(k+1)}\|_{\infty}\le A^{K-k+1} \frac{R^{\beta}}{\lambda}\}.$$ Now define $$F_{\gamma_k}^{(k)} = \sum_{T_{\gamma_k}\in \T_{\gamma_k}^g}\psi_{T_{\gamma_k}}F_{\gamma_k}^{(k+1)}, \quad F_{\gamma_{k-1}}^{(k)}=\sum_{\gamma_k \subset \gamma_{k-1}}F_{\gamma_k}^{(k)}.$$
    Property a) follows by noting that since $\{T_{\gamma_k}\}$ are a tiling of $\R^4$, for each $x\in \R^4$, either $|F^{(k)}_{\gamma_k}(x)|=0$ or $|F^{(k)}_{\gamma_k}(x)|\sim |F^{(k+1)}_{\gamma_k}(x)|.$
    The properties c), and the first bound in b) are immediate from the construction. The latter bound in b) follows from iterating property a). 
\end{proof}

It is not immediately clear which bound in b) is the "better" $L^{\infty}$ estimate. For some estimates, $\frac{R^\beta}{R_k}$ will suffice. However, as we will see in Lemma \ref{lem:upgrader}, the critical value of $\lambda$ we consider at this scale is $R_k^{1/2}R^{\beta/2}$ and so we morally have from b), $$\|F_{\gamma_k}^{(k)}\|_\infty \lesssim \frac{R^{\beta/2}}{R_k^{1/2}}= \left(\frac{R^{\beta}}{R_k}\right)^{1/2}.$$ This is the desired square-root cancellation bound.

\begin{lemma}\label{lem:L6andL12estimates}
 For any scale $R_k\le R^{1/2}$, we have
    $$\|F_{\gamma_k}^{(k)}\|_{L^{6}(\Omega_R)}\lessapprox |\Omega_R|^{\frac{1}{6}}\left(\frac{R^{\beta}}{R_k}\right)^{1/2}.$$
 For $R_k\le R^{1/3}$, we have $$\|F_{\gamma_k}^{(k)}\|_{L^{12}(\Omega_R)}\lessapprox \min \left\{\frac{R^{\frac{3}{4}\beta}}{\lambda^{\frac{1}{2}}R_k^{\frac{1}{4}}}, \frac{R^{\frac{3}{4}\beta}}{R^{\frac{1}{12}}R_k^{\frac{1}{2}}}\right\}|\Omega_R|^{\frac{1}{12}}.$$ 
\end{lemma}

\begin{proof}
    Note by a) of Lemma \ref{lem:PruningProperties} and $L^{6}$ cylindrical decoupling with Theorem \ref{thm:BDGmomentcurve}, we have \begin{align*}
        \|F_{\gamma_k}^{(k)}\|_{L^{6}(\Omega_R)}&\le C_\epsilon R^{\epsilon^{100}}\left(\sum_{\gamma_{k+1}\subset \gamma_k}\|F_{\gamma_{k+1}}^{(k)}\|_{L^{6}(\Omega_R)}^2\right)^{1/2}\\
        &\le C_\epsilon R^{\epsilon^{100}}\left(\sum_{\gamma_{k+1}\subset \gamma_k}\|F_{\gamma_{k+1}}^{(k+1)}\|_{L^{6}(\Omega_R)}^2\right)^{1/2}.
    \end{align*} We iterate this until $k=K$ to get $$\|F_{\gamma_k}^{(k)}\|_{L^{6}(\Omega_R)}\le C_\epsilon R^{\epsilon^{99}}\left(\sum_{\gamma\subset \gamma_k}\|F_{\gamma}\|_{L^{6}(\Omega_R)}^2\right)^{1/2}.$$ Since $\|F_\gamma\|_\infty\le 1$, this gives $$\|F_{\gamma_k}^{(k)}\|_{L^{6}(\Omega_R)}\lessapprox \left(\frac{R_k^{-1}}{R^{-\beta}}|\Omega_R|^{\frac{1}{3}}\right)^{1/2}= \left(\frac{R^\beta}{R_k}\right)^{\frac{1}{2}}|\Omega_R|^{\frac{1}{6}}.$$
    Similarly, note by a) of Lemma \ref{lem:PruningProperties} and $L^{12}$ cylindrical decoupling with Theorem \ref{thm:BDGmomentcurve}, we have $$\|F_{\gamma_k}^{(k)}\|_{L^{12}(\Omega_R)}\le C_\epsilon R^{\epsilon^{100}}\left(\sum_{\gamma_{k+1}\subset \gamma_k}\|F_{\gamma_{k+1}}^{(k+1)}\|_{L^{12}(\Omega_R)}^2\right)^{1/2}.$$ We iterate this until $k=k_0$ with $R_{k_0}\approx R^{1/3}$ to get $$\|F_{\gamma_k}^{(k)}\|_{L^{12}(\Omega_R)}\le C_\epsilon R^{\epsilon^{99}}\left(\sum_{\gamma_{k_0}\subset \gamma_{k}}\|F_{\gamma_{k_0}}^{(k_0)}\|_{L^{12}(\Omega_R)}^2\right)^{1/2}.$$ By b) in Lemma \ref{lem:PruningProperties} and the previous $L^6$ inequality with $k=k_0$, we have \begin{multline*}
\|F_{\gamma_{k_0}}^{(k_0)}\|_{L^{12}(\Omega_R)}^{12}\lesssim \min \left\{\frac{R^{6\beta}}{\lambda^6}, \frac{R^{6\beta}}{R^2}\right\}\|F_{\gamma_{k_0}}^{(k_0)}\|_{L^{6}(\Omega_R)}^{6}\\\lessapprox  \min \left\{\frac{R^{6\beta}}{\lambda^6}, \frac{R^{6\beta}}{R^2}\right\}|\Omega_R| \left(\frac{R^\beta}{R^{1/3}}\right)^{3} = \min\left\{\frac{R^{9\beta}}{\lambda^6R}, \frac{R^{9\beta}}{R^3}\right\}|\Omega_R|.
    \end{multline*}
    Therefore, we have \begin{align*}
        \|F_{\gamma_k}^{(k)}\|_{L^{12}(\Omega_R)}&\lessapprox \left(\frac{R_k^{-1}}{R^{-1/3}}\right)^{\frac{1}{2}}\min\left\{\frac{R^{\frac{3}{4}\beta}}{\lambda^{\frac{1}{2}}R^{\frac{1}{12}}}, \frac{R^{\frac{3}{4}\beta}}{R^{\frac{1}{4}}}\right\}|\Omega_R|^{\frac{1}{12}}\\
        &=\min \left\{\frac{R^{\frac{3}{4}\beta+\frac{1}{12}}}{\lambda^{\frac{1}{2}}R_k^{\frac{1}{2}}}, \frac{R^{\frac{3}{4}\beta}}{R^{\frac{1}{12}}R_k^{\frac{1}{2}}}\right\}|\Omega_R|^{\frac{1}{12}}.
    \end{align*}
We can actually do better than the first bound by going straight to $L^6$. Indeed, $$\|F_{\gamma_k}^{(k)}\|_{L^{12}(\Omega_R)}^{12}\lesssim \|F_{\gamma_k}^{(k)}\|_{\infty}^6\|F_{\gamma_k}^{(k)}\|_{L^{6}(\Omega_R)}^6 \lessapprox \frac{R^{6\beta}}{\lambda^6}\left(\frac{R^\beta}{R_k}\right)^{3}|\Omega_R|=\frac{R^{9\beta}}{\lambda^6R_k^3}|\Omega_R|. $$
\end{proof}

\section{Square Function Stopping Time}\label{sec:stopTime}
Let $$g_k(x)=\sum_{\gamma_k}|F_{\gamma_k}^{(k+1)}(x)|^2.$$ We perform a stopping time decomposition on $\Omega_R$. The decomposition has two regimes: $R_k>R^\alpha$ and $R_k\le R^{\alpha}$. In the large scale regime $R_k>R^\alpha$, we will split into boxes of sidelengths $$R_k \times R_k \times R_k \times R^\alpha $$ while for the small scale regime $R_k\le R^\alpha$, we will split into cubes of sidelength $R_k$. We will call both "boxes" of sidelength $R_k$; however, the meaning of "boxes" is as above, depending on the size of $R_k$ relative to $R^\alpha$.
Split $\Omega_R$ into "boxes" of sidelength $R_{K-1}$. We say one of these "boxes" $B$ is $(K-1)$-bad if $$\frac{1}{|B|}\int_B g_{K-1}^2 \ge A^2R^{2\beta}.$$ Let $\Lambda_{K-1}$ be the union of the $(K-1)$-bad "boxes". We form $\Lambda_k$ similarly. We split the remaining "boxes" after forming $\Lambda_{k+1}$ into smaller "boxes" of sidelength $R_k$. We say one of these "boxes" $B$ is $k$-bad if \begin{equation}\label{eq:k-badDefn}
    \frac{1}{|B|}\int_B g_{k}^2 \ge A^{2K-2k}R^{2\beta}.
\end{equation} Let $\Lambda_k$ be the union of these $k$-bad "boxes". Finally, let $\Lambda_0$ be the union of the cubes with sidelength $1$ which remain after forming $\Lambda_1$. Note by H\"{o}lder's inequality, $$g_k\le R^{\epsilon} g_{k+1}.$$

\begin{lemma}\label{lem:goodbadDicotomy} For $k\ge 1$, we have
$$\int_{\Lambda_k}|g_{k}|^2\ge A^{2K-2k}R^{2\beta}|\Lambda_k|, \quad \int_{\Lambda_k}|g_{k+1}|^2\lesssim A^{2K-2k-2}R^{2\beta}|\Lambda_{k}|.$$  Moreover, for $x\in \Lambda_k$, $$g_k(x)\approx A^{K-k}R^{\beta}$$ and for $m>k$, $$g_m(x)\lesssim A^{K-m}R^\beta.$$ For $k=0$, the second inequality still holds and for $x\in \Lambda_0$, $g_0(x)\lesssim R^{O(\epsilon)}R^\beta.$ 
\end{lemma}

\begin{proof}
    The first inequality follows from summing over all $R_{k}$-"boxes" $B\subset \Lambda_k$ with the definition \eqref{eq:k-badDefn} of $B\subset \Lambda_k$. For the second inequality, first observe that $g_k$ is locally constant on balls of radius $R_k$; indeed, $g_k$ is Fourier supported on a ball of radius $O(R_k^{-1})$ about the origin. This follows since each $|F_{\gamma_k}^{(k+1)}|^2$ has Fourier support on $C(\gamma_k-\gamma_k)$ by Lemma \ref{lem:PruningProperties}. Now note that if $B\subset \Lambda_k$, then $B$ must lie in a $(k+1)$-good "box". Therefore, for $B\subset \Lambda_k$, $$\int_{B}|g_{k+1}|^2 \lesssim \frac{|B|}{|B^\prime|}\int_{B^\prime}|g_{k+1}|^2\le A^{2K-2k-2}R^{2\beta}|B|,$$ where $B^\prime$ is the $(k+1)$-good parent of $B$. Now, we sum over $B\subset \Lambda_k$ to finish. 
    For the pointwise estimate, we use H\"{o}lder's inequality as above and the locally constant property. Further, for $m>k$, for the $R_m$-"box" containing $x$, call it $B$, we have $$\int_B |g_m|^2\le A^{2K-2m}R^{2\beta}|B|.$$ By the locally constant property, $g_m(x)\lesssim A^{K-m}R^{\beta}$. 
\end{proof}

We relate each $\Lambda_k$ to the pruning, identical to \cite[Lemma 5]{Guth_2024}. Define $$F_{\tau_i}^{(k+1)}= \sum_{\gamma_k\subset \tau_i}F_{\gamma_k}^{(k+1)}.$$

\begin{lemma}\label{lem:PassToPruned}
    For each $\tau_i$ and each $x\in \Lambda_k$, $$|F_{\tau_i}(x) - F_{\tau_i}^{(k+1)
    }(x)|\le \frac{\lambda}{1000}.$$
\end{lemma}

\begin{proof}
    Fix $\tau_i$. Write $$F_{\tau_i} = F^{(k+1)}_{\tau_i} + \sum_{m=k+1}^{K-1}(F_{\tau_i}^{(m+1)}-F_{\tau_i}^{(m)}).$$ Further for each term, we have $$F_{\tau_i}^{(m+1)}-F_{\tau_i}^{(m)}= \sum_{\gamma_m\subset \tau_i}F_{\gamma_m}^{(m+1)}-F_{\gamma_m}^{(m)}.$$ We bound each term of this sum. Let $\T^b_{\gamma_m}$ be the set of bad wavepackets as in the proof of Lemma \ref{lem:PruningProperties}.  Note that since $|F_{\gamma_m}^{(m+1)}|$ is locally constant on each $T_{\gamma_m}$, we have \begin{align*}
        |F_{\gamma_m}^{(m+1)}-F_{\gamma_m}^{(m)}| &= \sum_{T_{\gamma_m}\in \T^b_{\gamma_m}}\psi_{T_{\gamma_m}}|F_{\gamma_m}^{(m+1)}|\\
        &\lesssim \sum_{T_{\gamma_m}\in \T^b_{\gamma_m}}A^{-(K-m+1)}\frac{\lambda}{R^{\beta}}\psi_{T_{\gamma_m}}|F_{\gamma_m}^{(m+1)}|^2\\
        &\lesssim A^{-(K-m+1)}\frac{\lambda}{R^{\beta}}|F_{\gamma_m}^{(m+1)}|^2.
    \end{align*} Therefore, by the triangle inequality, $$|F_{\tau_i}^{(m+1)}-F_{\tau_i}^{(m)}|\lesssim A^{-(K-m+1)}\frac{\lambda}{R^{\beta}}g_m.$$ So, for $x\in \Lambda_k$, by Lemma \ref{lem:goodbadDicotomy}, we have $g_m(x)\lesssim A^{K-m}R^{\beta}$. Now, we may choose $A$ sufficiently large such that we get the desired inequality.
\end{proof}

We also have the following bound on $\lambda$, which is analogous to \cite[Lemma 4.1]{maldague2024smallcapdecouplingmoment}. This is the step where the $4$-broad formulation is needed; passing to the pruned $F_{\tau_i}^{(k+1)}$ is not available for the $4$-linear version.

\begin{lemma}\label{lem:upgrader}
    If $U_{\lambda}\cap \Lambda_k$ is nonempty, then $$\lambda\lesssim R^{O(\epsilon)}R_k^{1/2}R^{\beta/2}.$$
\end{lemma}

\begin{proof}
    Say, $x\in U_{\lambda}\cap B$ where $B\subset \Lambda_k$ is a "box" with sidelength $R_k$. By Lemma \ref{lem:PassToPruned}, we may pass to $F_{\tau_i}^{(k+1)}= \sum_{\gamma_k\subset \tau_i}F_{\gamma_k}^{(k+1)}$: $$\lambda\le \min_{1\le i\le 4}|F_{\tau_i}(x)|\le \frac{\lambda}{1000}+\min_{1\le i\le 4}|F_{\tau_i}^{(k+1)}(x)|.$$   Then, using Lemma \ref{lem:goodbadDicotomy}, $$\lambda\lesssim \min_{1\le i\le 4}|F_{\tau_i}^{(k+1)}(x)|\le\sum_{\gamma_k}|F_{\gamma_k}^{(k+1)}(x)|\le R_k^{1/2}g_k^{1/2}(x)\lesssim R^{O(\epsilon)}R_k^{1/2}R^{\beta/2}.$$
\end{proof}

\section{Multilinear Restriction Estimates}\label{sec:multRestriction}
Recall from \eqref{eq:4-linDyadicPig}, it suffices to show $$\lambda^p |U_{\lambda}|= \sum_{k}\lambda^p |U_{\lambda}\cap \Lambda_k|\lessapprox R^{\frac{p}{2}\beta}|\Omega_R|.$$ We will eventually get a uniform estimate for $\lambda^{p}|U_{\lambda}\cap \Lambda_k|$, for each $k$. In this section, we will get an $L^8$ estimate on $\Lambda_k$.

\begin{lemma}
    \begin{equation}\label{eq:restrictionConclusion}
    \lambda^8|U_{\lambda}\cap \Lambda_k| \lesssim R^{O(\epsilon)}\max\left\{1, \frac{R_k}{R^\alpha}\right\}R^{4\beta}|\Lambda_k|.
\end{equation}
\end{lemma}
\noindent By Lemma \ref{lem:PassToPruned}, we can pass to $F_{\tau_i}^{(k+1)}= \sum_{\gamma_k\subset \tau_i}F_{\gamma_k}^{(k+1)}$, just as in the proof of Lemma \ref{lem:upgrader}. Moreover, we can bound by the geometric mean $$\min_{1\le i\le 4}|F_{\tau_i}^{(k+1)}|\le \prod_{i=1}^4|F_{\tau_i}^{(k+1)}|^{1/4}.$$ We split into two regimes. 

First, consider $R_k\le R^\alpha$. By the $L^8$ multilinear restriction estimate Theorem \ref{thm:MultRestrictionMomentCurve}, $$\lambda^8 |U_{\lambda}\cap \Lambda_k| \lesssim \sum_{B=B_{R_k}\subset \Lambda_k}\int_{B} \prod_{i=1}^4 |F_{\tau_i}^{(k+1)}|^2 \lesssim \sum_{B\subset \Lambda_k}\frac{1}{|B|^3}\prod_{i=1}^4 \left(\int_B |F_{\tau_i}^{(k+1)}|^2\right).$$ By $L^2$ orthogonality and H\"{o}lder's, $$\int_B |F_{\tau_i}^{(k+1)}|^2\lesssim \int_B \sum_{\gamma_k}|F_{\gamma_k}^{(k+1)}|^2= \int_B g_k\lesssim |B|^{1/2}\left(\int_{B}|g_k|^2\right)^{1/2}\lesssim |B|R^{\beta+O(\epsilon)}.$$ Therefore, we have \begin{equation*}\label{eq:L^8cube}
    \lambda^8|U_{\lambda}\cap \Lambda_k|\lesssim R^{4\beta+O(\epsilon)}|\Lambda_k|.
\end{equation*}

Now, consider $R^\alpha<R_k\le R^\beta$. If we ignore $x_4$ and only use the $L^6$ trilinear restriction estimate Theorem \ref{thm:MultRestrictionMomentCurve}, we get by the same argument as in the previous case that \begin{equation}\label{eq:L^6restriction}
    \lambda^6|U_{\lambda}\cap \Lambda_k| \lesssim R^{3\beta+O(\epsilon)}|\Lambda_k|.
\end{equation} However, we can do better; we can apply $L^8$ $4$-linear estimate Theorem \ref{thm:MultRestrictionMomentCurve} at the scale $R^\alpha$ and using the locally constant property, we can then further use the $L^6$ trilinear estimate Theorem \ref{thm:MultRestrictionMomentCurve}. This is similar to the argument in \cite[Section 4.1]{maldague2024smallcapdecouplingmoment}. This gives $$\lambda^8 |U_\lambda\cap \Lambda_k| \lesssim \sum_{B\subset \Lambda_k}\sum_{B_{R^\alpha}\subset B}\int_{B_{R^\alpha}} \prod_{i=1}^4 |F_{\tau_i}^{(k+1)}|^2 \lesssim \sum_{B_{R^\alpha}\subset \Lambda_k}\frac{1}{|B_{R^\alpha}|^3}\prod_{i=1}^4 \left(\int_{B_{R^\alpha}} |F_{\tau_i}^{(k+1)}|^2\right).$$ By $L^2$ orthogonality and H\"{o}lder's and the locally constant property of $g_k$, \begin{multline*}
    \int_{B_{R^\alpha}} |F_{\tau_i}^{(k+1)}|^2\lesssim \int_{B_{R^\alpha}} \sum_{|\gamma^\prime|=R^{-\alpha}, \gamma^\prime \subset \tau_i}|F_{\gamma^\prime}^{(k+1)}|^2\lesssim R_kR^{-\alpha}\int_{B_{R^\alpha}} \sum_{\gamma_k}|F_{\gamma_k}^{(k+1)}|^2\\= R_kR^{-\alpha}\int_{B_{R^\alpha}} g_k\lesssim R_kR^{-\alpha}|B_{R^\alpha}|^{1/2}\left(\int_{B_{R^\alpha}}|g_k|^2\right)^{1/2}\lesssim R_kR^{-\alpha}|B_{R^\alpha}|R^{\beta+O(\epsilon)}.
\end{multline*}
Therefore, we have $$\lambda^8|U_{\lambda}\cap \Lambda_k|\lesssim R_kR^{\beta - \alpha+O(\epsilon)}\sum_{B_{R^\alpha}\subset \Lambda_k}\frac{1}{|B_{R^\alpha}|^2}\prod_{i=1}^3 \left(\int_{B_{R^\alpha}}\sum_{|\gamma^\prime|=R^{-\alpha}, \gamma^\prime \subset \tau_i}|F_{\gamma^\prime}^{(k+1)}|^2\right).$$ Similarly to the local constancy of $g_k$, we have that $$\sum_{|\gamma^\prime|=R^{-\alpha}}|F_{\gamma^\prime}^{(k+1)}|^2$$ are roughly constant on cubes $B_{R^\alpha}$. Hence, $$\lambda^8|U_{\lambda}\cap \Lambda_k|\lesssim R_kR^{\beta - \alpha+O(\epsilon)}\sum_{B\subset \Lambda_k} \int_{B}\prod_{i=1}^3\left(\sum_{|\gamma^\prime|=R^{-\alpha}, \gamma^\prime \subset \tau_i}|F_{\gamma^\prime}^{(k+1)}|^2\right).$$ By the $L^6$ trilinear restriction estimate Theorem \ref{thm:MultRestrictionMomentCurve}, we have $$\int_{B}\prod_{i=1}^3\left(\sum_{|\gamma^\prime|=R^{-\alpha}, \gamma^\prime \subset \tau_i}|F_{\gamma^\prime}^{(k+1)}|^2\right)\lesssim \frac{1}{|B|^2}\prod_{i=1}^3 \left(\int_B \sum_{|\gamma^\prime|=R^{-\alpha}}|F_{\gamma^\prime}^{(k+1)}|^2\right).$$ By the same orthogonality argument as above, we have $$\lambda^8|U_{\lambda}\cap \Lambda_k| \lesssim R_kR^{4\beta-\alpha+O(\epsilon)}|\Lambda_k|.$$ In summary, \begin{equation*}
    \lambda^8|U_{\lambda}\cap \Lambda_k| \lesssim R^{O(\epsilon)}|\Lambda_k|\begin{cases}
        R^{4\beta}, & R_k\le R^\alpha\\
        R_kR^{4\beta-\alpha},& R_k> R^\alpha
    \end{cases}
    =R^{O(\epsilon)}\max\left\{1, \frac{R_k}{R^\alpha}\right\}R^{4\beta}|\Lambda_k|.
\end{equation*}
For clarity moving forward, we will omit $R^{O(\epsilon)}$ as it is harmless for our argument.

\section{Adding exponents: from $L^8$ to $L^p$}\label{sec: 8+4=12}

The $L^{p}$ estimate $$\lambda^p|U_\lambda\cap \Lambda_k|\lessapprox R^{\frac{p}{2}\beta}|\Omega_R|$$ follows by combining \eqref{eq:restrictionConclusion} with an improved estimate for $|\Lambda_k|$.
We can first quickly dispense with $k=0$. Indeed, by Lemma \ref{lem:upgrader} and \eqref{eq:restrictionConclusion}, we use the trivial estimate $|\Lambda_0|\le |\Omega_R|$ and we are done. Now, we focus on $k\ge 1$.
To get such an estimate, we decompose $g_k$ into a high and low piece. Define the adjacency relation $\gamma_{k}\sim \gamma_{k}^\prime$ if $\gamma_k, \gamma_k^\prime$ are the same or adjacent caps. Let $$g_k^h = \sum_{\gamma_k}\sum_{\gamma_{k+1}\not\sim \gamma_{k+1}^\prime\subset \gamma_k}F_{\gamma_{k+1}}^{(k+1)}\ol{F_{\gamma_{k+1}^\prime}^{(k+1)}},$$ and $$g_k^l = \sum_{\gamma_k}\sum_{\gamma_{k+1}\sim \gamma_{k+1}^\prime\subset \gamma_k}F_{\gamma_{k+1}}^{(k+1)}\ol{F_{\gamma_{k+1}^\prime}^{(k+1)}}.$$
By Lemma \ref{lem:goodbadDicotomy}, we have $$A^{2K-2k}R^{2\beta}|\Lambda_k|\lesssim \int_{\Lambda_k}|g_k|^2\lesssim \int_{\Lambda_k}|g_k^l|^2+\int_{\Lambda_k}|g_k^h|^2.$$  We will show $g_k^l$ is comparable to $g_{k+1}$.
    Note that $$g_k^l(x) \le  g_{k+1}(x)+ \left|\sum_{\gamma_k}\sum_{\substack{\gamma_{k+1}\sim \gamma_{k+1}^\prime \subset \gamma_k\\ \gamma_{k+1}, \gamma_{k+1}^\prime adjacent}}F^{(k+1)}_{\gamma_{k+1}}(x)\ol{F^{(k+1)}_{\gamma_{k+1}^\prime}(x)}\right|.$$ Thus, by Cauchy-Schwarz, $$|g_k^l(x)|\le g_{k+1}(x)+\sum_{\gamma_k}\sum_{\substack{\gamma_{k+1}, \gamma_{k+1}^\prime \subset \gamma_k\\ \gamma_{k+1}, \gamma_{k+1}^\prime adjacent}}\frac{|F^{(k+1)}_{\gamma_{k+1}}(x)|^2+|F^{(k+1)}_{\gamma_{k+1}^\prime}(x)|^2}{2}\lesssim g_{k+1}(x).$$
By Lemma \ref{lem:goodbadDicotomy}, we have $$\int_{\Lambda_k}|g_k^{l}|^2 \lesssim \int_{\Lambda_k}|g_{k+1}|^2\lesssim A^{2K-2k-2}R^{2\beta}|\Lambda_k|.$$
Therefore, we can choose $A$ sufficiently large such that $$\int_{\Lambda_k}|g_k^{l}|^2 \le \frac{1}{2}A^{2K-2k}R^{2\beta}|\Lambda_k|.$$ Hence, the high piece dominates so that $$R^{2\beta}|\Lambda_k|\lesssim \int_{\Lambda_k}|g_k^h|^2\le |\Lambda_k|^{1-\frac{2}{q}}\left(\int_{\Omega_R}|g_k^h|^q\right)^{2/q}$$ which gives for $q\ge 2$, $$R^{\beta q}|\Lambda_k|\lesssim \int_{\Omega_R}|g_k^h|^q.$$ From this point further, we ignore $x_4$ and only use $3$-dimensional considerations; namely, the decoupling of the cone into planks Theorems \ref{thm:BDconedec} and \ref{thm:small-capCone}. We use these decouplings of the cone to get the following improved orthogonality for $g_k^h$. The high piece has frequency supported on differences of separated $\gamma_{k+1}$. After rescaling, these difference sets are roughly cone planks in the first three frequency variables. We refer to \cite[Lemma 7, Lemma 8]{Guth_2024}.

\begin{lemma}\label{lem:highOrthog}
For $4\le \frac{2}{\beta}\le q\le 6$, and $R_k=R^{2/q}\in [R^{1/3}, R^{\beta}]$, 
    $$\int_{\Omega_R}|g_k^h|^{q} \lessapprox R_k^{-1}R \sum_{\gamma_{k+1}}\int_{\Omega_R} |F_{\gamma_{k+1}}^{(k+1)}|^{2q}.$$  For $R_k\le R^{1/3}$, $$\int_{\Omega_R}|g_k^h|^6\lessapprox \left(\sum_{\gamma_{k+1}}\|F_{\gamma_{k+1}}^{(k+1)}\|_{L^{12}(\Omega_R)}^4\right)^3.$$
\end{lemma}

\begin{proof}
    We deviate from the presentation in \cite[Section 2.3-2.4]{Guth_2024} for clarity. In particular, their arguments are completely geometric. We reformulate in terms of the extension operator, and simplify the argument down to a simple sequence of variable changes. Fix $x_4$ and only consider $\R^3$ as the first three variables; integrating over $x_4$ and Minkowski's inequality at the end will give the inequalities over $\Omega_R$.
    
    Let $E$ be the extension operator associated with the moment curve in $\R^3$, i.e. for any interval $J\subset [\frac{1}{2}, 1]$ and $f_J:J\to \C$, $$Ef_J(x)= \int_{[\frac{1}{2}, 1]}f_J(t) e(x\cdot (t, t^2, t^3))\;dt.$$ Let $J_k$ be the interval of length $R_k^{-1}$ associated with the cap $\gamma_k$. Consider the analog of the high piece $$G(x)=\sum_{J_k}\sum_{J_{k+1}\not\sim J_{k+1}^\prime\subset J_k} Ef_{J_{k+1}}(x)\ol{Ef_{J_{k+1}^\prime}}(x).$$ After an affine rescaling, we will see that for each $J_k$, $$G_{J_k}(x)=\sum_{J_{k+1}\not\sim J_{k+1}^\prime\subset J_k} Ef_{J_{k+1}}(x)\ol{Ef_{J_{k+1}^\prime}}(x)$$ is Fourier supported on a thin neighborhood of a cone plank. Indeed, by two changes of variables, \begin{align*}
        G_{J_k}(x) & = \sum_{J_{k+1}\not\sim J_{k+1}^\prime\subset J_k} \int_{J_{k+1}\times J_{k+1}^\prime}f(t)\ol{f(s)}e(x\cdot (t-s, t^2-s^2, t^3-s^3))\; dtds\\
        &= \frac{1}{2}\sum_{\substack{J_{k+1}\not\sim J_{k+1}^\prime\\\subset J_k}} \int_{[R_{k+1}^{-1}, R_k^{-1}]} \int_{J_{k+1}+J_{k+1}^\prime} g(u,v)e(x\cdot(u, uv, \frac{u^3}{4}+\frac{3}{4}uv^2))\; dvdu\\
        &= \frac{1}{2R_k}\sum_{\substack{J_{k+1}\not\sim J_{k+1}^\prime\\\subset J_k}}\int_{[R^{-\epsilon}, 1]} \int_{J_{k+1}+J_{k+1}^\prime}g(R_ku^\prime, v)e(\frac{1}{R_k}x\cdot(u^\prime, u^\prime v, \frac{(u^\prime)^3}{4R_k^2}+\frac{3}{4}u^\prime v^2))\; dvdu^\prime,
    \end{align*} for some function $g$. Note $J_{k+1}+J_{k+1}^\prime\subset 2J_k$. Since $\frac{(u^\prime)^3}{4R_k^2}\lesssim \frac{1}{R_k^2}$, the surface $$\{(u^\prime, u^\prime v, \frac{(u^\prime)^3}{4R_k^2}+\frac{3}{4}u^\prime v^2)): u^\prime \in [R^{-\epsilon}, 1], v\in 2J_k\}$$ is contained in a $\sim 1\times R_k^{-1}\times R_k^{-2}$ cone plank. Since $u^\prime$ ranges over $[R^{-\epsilon}, 1]$, we technically can dyadic pigeonhole and rescale $u^\prime$; however, this only introduces $R^{O(\epsilon)}$ losses which are harmless. With the change of variables $y=\frac{1}{R_k}x$, we can apply a decoupling for the cone on a ball of $\frac{R}{R_k}$. For $R_k\le R^{1/3}$, we may use canonical decoupling on smaller balls of radius $R_k^2$. Indeed by Theorem \ref{thm:BDconedec} this gives\begin{align*}
       \int_{B_R}|G(x)|^6\; dx=R_k^3\int_{B_{R/R_k}}|\tilde{G}(y)|^6\; dy
       &=R_k^3\sum_{B_{R_k^2}\subset B_{R/R_k}}\|\tilde{G}\|_{L^6(B_{R_k^2})}^6\\
       &\lessapprox R_k^3\sum_{B_{R_k^2}\subset B_{R/R_k}} \left(\sum_{J_k}\|\tilde{G}_{J_k}\|_{L^6(B_{R_k^2})}^2\right)^{3}.
    \end{align*} Now, by undoing the variables $y=\frac{1}{R_k}x$ and Minkowski's inequality, we get $$\int_{B_R}|G(x)|^6\; dx\lessapprox \left(\sum_{J_k}\|G_{J_k}\|_{L^6(B_{R})}^2\right)^{3}\lesssim R^{O(\epsilon)} \left(\sum_{J_{k+1}}\|Ef_{J_{k+1}}\|_{L^{12}(B_R)}^4\right)^{3}.$$ By a standard $R^{-1}$-thickening and integrating over $x_4$, we get the desired inequality above. For $R_k>R^{1/3}$, we use the small-cap decoupling Theorem \ref{thm:small-capCone} on a ball of radius $\frac{R}{R_k}$ with $\kappa$ such that $\left(\frac{R}{R_k}\right)^{-\kappa}= \frac{1}{R_k}$. This gives $R_k= R^\frac{\kappa}{\kappa+1}= R^{2/q}$. Moreover, the cone planks as described above fit inside the small-cap cone planks since $\frac{1}{R_k^2} \le \left(\frac{R}{R_k}\right)^{-1}$ in this regime. By the same computation as in the previous case, we get the desired inequality.
\end{proof}

Now, we use the previous Lemma \ref{lem:highOrthog} to get the improved estimate for $\Lambda_k$.

\begin{corollary}\label{cor:Lambda_kEstimate}
    For $4\le \frac{2}{\beta}\le q\le 6$, and $R_k=R^{2/q}\in [R^{1/3}, R^{\beta}]$, 
    $$|\Lambda_k|\lessapprox\min\left\{\frac{R^{\beta(q-3)+1}}{R_k^3\lambda^{2q-6}}, \frac{R^{\beta(q-3)+1}}{R_k^{2q-3}}\right\}|\Omega_R|.$$ For $R_k\le R^{1/3}$, $$|\Lambda_k|\lessapprox \min\left\{\frac{R^{3\beta}}{\lambda^6}, \frac{R^{3\beta}}{R_k^3R}\right\}|\Omega_R|.$$
\end{corollary}

\begin{proof}
For $R_k= R^{2/q}$ for $q\in [\frac{2}{\beta},6]$, then by Lemma \ref{lem:L6andL12estimates}\begin{align*}
    R^{\beta q}|\Lambda_k|\lesssim\int_{\Omega_R}|g_k^h|^q&\lessapprox R_k^{-1}R \sum_{\gamma_{k+1}}\int_{\Omega_R} |F_{\gamma_{k+1}}^{(k+1)}|^{2q}\\
    &\lesssim R_k^{-1}R\left(\frac{R^\beta}{\lambda}\right)^{2q-6}\sum_{\gamma_{k+1}}\int_{\Omega_R}|F_{\gamma_{k+1}}^{(k+1)}|^{6}\\
    &\lessapprox R_k^{-1}R \left(\frac{R^\beta}{\lambda}\right)^{2q-6} R_k\left(\frac{R^{\beta}}{R_k}\right)^3|\Omega_R|\\
    &= \frac{R^{\beta(2q-3)+1}}{R_k^3\lambda^{2q-6}}|\Omega_R|.
\end{align*} The second bound in this case is found by repeating the same computation but using Lemma \ref{lem:PruningProperties} $$\|F_{\gamma_{k+1}}^{(k+1)}\|_\infty \lesssim \frac{R^{\beta}}{R_k}.$$
For $R_k\le R^{1/3}$, \begin{align*}
    R^{6\beta}|\Lambda_k|\lesssim\int_{\Omega_R}|g_k^h|^6&\lessapprox \left(\sum_{\gamma_{k+1}}\|F_{\gamma_{k+1}}^{(k+1)}\|_{L^{12}(\Omega_R)}^4\right)^3\\
    &\lessapprox  \left(\sum_{\gamma_{k+1}}\min \left\{\frac{R^{3\beta}}{\lambda^2R_k}, \frac{R^{3\beta}}{R^{\frac{1}{3}}R_k^2}\right\}|\Omega_R|^{\frac{1}{3}}\right)^3\\
    &\lessapprox \min \left\{\frac{R^{9\beta}}{\lambda^6}, \frac{R^{9\beta}}{RR_k^3}\right\}|\Omega_R|
\end{align*}

\end{proof}

\section{Proof of Theorem \ref{thm:p<11}}\label{sec:proof}
Fix $p\in (10, 14]$ with $\alpha = \left(\frac{p-6}{2}\right)\beta-1$, $\max\{\frac{1}{3}, \frac{2}{p-6}\}\le \beta\le \min\{\frac{1}{2},\frac{2}{p-8}\}$. The main strategy is to combine \eqref{eq:restrictionConclusion} with Corollary \ref{cor:Lambda_kEstimate}. In some cases, we do not need the full strength of $\eqref{eq:restrictionConclusion}$, only the $L^6$ estimate \eqref{eq:L^6restriction}. Cases Ia and Ib will verify for the smallest scales $R_k\le R^{1/3}$. Case II verifies for the intermediate scales and Case III verifies for the largest scales. 

\subsection{Case Ia: $R_k\le R^{1/3}$, $p\le 12$}

By Corollary \ref{cor:Lambda_kEstimate} and \eqref{eq:L^6restriction}, we get $$\lambda^6 |U_\lambda \cap \Lambda_k|\lessapprox R^{3\beta}\frac{R^{3\beta}}{\lambda^6}|\Omega_R|$$ which gives $$\lambda^{12} |U_\lambda \cap \Lambda_k|\lessapprox R^{6\beta}|\Omega_R|.$$ By our assumption in Section \ref{sec:Outline} that $\lambda\ge R^{\beta/2}$, we get $$\lambda^{p} |U_\lambda \cap \Lambda_k|\lessapprox R^{\frac{p}{2}\beta}|\Omega_R|.$$

\subsection{Case Ib: $R_k \le R^{1/3}$, $p>12$} If $R_k\le R^\alpha$, then by Corollary \ref{cor:Lambda_kEstimate} and \eqref{eq:restrictionConclusion}, we get $$\lambda^8 |U_\lambda \cap \Lambda_k|\lessapprox R^{4\beta}\frac{R^{3\beta}}{\lambda^6}|\Omega_R|.$$ By our assumption in Section \ref{sec:Outline} that $\lambda\ge R^{\beta/2}$, we get $$\lambda^{p} |U_\lambda \cap \Lambda_k|\lessapprox R^{\frac{p}{2}\beta}|\Omega_R|.$$
Alternatively, now suppose $R_k>R^\alpha$ if any scales exist. By Corollary \ref{cor:Lambda_kEstimate} and \eqref{eq:restrictionConclusion}, we get $$\lambda^8 |U_\lambda \cap \Lambda_k|\lessapprox R^{4\beta}\frac{R^{3\beta-\alpha}}{R_k^2 R}|\Omega_R|.$$ By Lemma \ref{lem:upgrader}, this gives \begin{align*}
    \lambda^p |U_\lambda \cap \Lambda_k|&\lessapprox R^{\frac{p}{2}\beta}\frac{R^{\beta(7-\frac{p}{2})-\alpha-1}}{R_k^2}\lambda^{p-8}|\Omega_R|\\
    &\lessapprox R^{\frac{p}{2}\beta}\frac{R^{\beta(7-\frac{p}{2})-\alpha-1}}{R_k^2}(R_k^{\frac{1}{2}}R^{\frac{\beta}{2}})^{p-8}|\Omega_R|\\
    &= R^{\frac{p}{2}\beta}\frac{R_k^{\frac{p}{2}-6}}{R^{\alpha+1-3\beta}}|\Omega_R|.
\end{align*} Note $\alpha+1 = \beta\left(\frac{p}{2}-3\right)$. Since $p>12$, this gives $$\frac{R_k^{\frac{p}{2}-6}}{R^{\alpha+1-3\beta}}= \left(\frac{R_k}{R^\beta}\right)^{\frac{p}{2}-6}\le 1.$$ Hence, we get $\lambda^{p} |U_\lambda \cap \Lambda_k|\lessapprox R^{\frac{p}{2}\beta}|\Omega_R|.$

\subsection{Case II: $R^{1/3}<R_k<R^{\frac{4}{p-2}}$} Let $R_k= R^{\frac{2}{q}}$ with $\frac{p-2}{2}\le q\le 6$. We split into two subcases depending on the size of $\lambda$. First, suppose $\lambda\ge R_k$. By Corollary \ref{cor:Lambda_kEstimate} and \eqref{eq:restrictionConclusion}, we get $$\lambda^8|U_\lambda\cap \Lambda_k|\lessapprox R^{4\beta}\frac{R^{\beta(q-3)+1}}{R_k^3\lambda^{2q-6}}|\Omega_R|\max\left\{1, \frac{R_k}{R^\alpha}\right\}.$$ Since $q\ge \frac{p-2}{2}$, we get \begin{align*}
    \lambda^p|U_\lambda\cap \Lambda_k|&\lessapprox R^{\frac{p}{2}\beta}\frac{R^{\beta(q+1-\frac{p}{2})+1}}{R_k^3\lambda^{2q-p+2}}|\Omega_R|\max\left\{1, \frac{R_k}{R^\alpha}\right\}\\
    &\le R^{\frac{p}{2}\beta}\frac{R^{\beta(q+1-\frac{p}{2})+1}}{R_k^3R_k^{2q-p+2}}|\Omega_R|\max\left\{1, \frac{R_k}{R^\alpha}\right\}\\
    &=R^{\frac{p}{2}\beta}\frac{R^{\beta(q+1-p)+\frac{p}{2}\beta+1}}{R_k^{2q-p+5}}|\Omega_R|\max\left\{1, \frac{R_k}{R^\alpha}\right\}\\
    &=R^{\frac{p}{2}\beta}\frac{R^{\beta(q+1-p)+3\beta+\alpha+2}}{R_k^{2q-p+5}}|\Omega_R|\max\left\{1, \frac{R_k}{R^\alpha}\right\}\\
    &= R^{\frac{p}{2}\beta}\frac{R^{\beta(q+4-p)}}{R_k^{q-p+4}}\cdot \frac{R^2R^\alpha}{R_k^{q+1}}|\Omega_R|\max\left\{1, \frac{R_k}{R^\alpha}\right\}.
\end{align*} Note $R_k^q = R^2$ and that $$\frac{R^\alpha}{R_k}\max\left\{1, \frac{R_k}{R^\alpha}\right\}\le \max\left\{\frac{R^\beta}{R_k}, 1\right\}.$$ For $p\ge 11\ge q+5$, we have $$\left(\frac{R_k}{R^\beta}\right)^{p-q-5}\le 1.$$ These give $\lambda^p|U_\lambda\cap \Lambda_k|\lessapprox R^{\frac{p}{2}\beta}|\Omega_R|.$ Also for $p<11$, note $\alpha\le \frac{1}{4}<\frac{1}{3}$ so, $R_k>R^{\alpha}$ and our estimate becomes $$\lambda^p|U_\lambda\cap \Lambda_k|\lessapprox R^{\frac{p}{2}\beta}|\Omega_R| \left(\frac{R_k}{R^\beta}\right)^{p-q-4}\le R^{\frac{p}{2}\beta}|\Omega_R|. $$
Alternatively, suppose $\lambda\le R_k$.  By Corollary \ref{cor:Lambda_kEstimate} and \eqref{eq:restrictionConclusion}, we get $$\lambda^8|U_\lambda\cap \Lambda_k|\lessapprox R^{4\beta}\frac{R^{\beta(q-3)+1}}{R_k^{2q-3}}|\Omega_R|\max\left\{1, \frac{R_k}{R^\alpha}\right\}.$$ Therefore, $$\lambda^p|U_\lambda\cap \Lambda_k|\lessapprox R^{\frac{p}{2}\beta}\frac{R^{\beta(q+1-\frac{p}{2})+1}}{R_k^{2q-p+5}}|\Omega_R|\max\left\{1, \frac{R_k}{R^\alpha}\right\}$$ which is the same estimate as above so we proceed the same and get $\lambda^p|U_\lambda\cap \Lambda_k|\lessapprox R^{\frac{p}{2}\beta}|\Omega_R|.$

\subsection{Case III: $R_k> R^{\frac{4}{p-2}}$} Let $R_k= R^{\frac{2}{q}}$ with $\frac{2}{\beta}\le q< \frac{p-2}{2}$ if any such scales exist. As in Case II, we have $$\lambda^p|U_\lambda\cap \Lambda_k|\lessapprox R^{\frac{p}{2}\beta}\frac{R^{\beta(q+1-\frac{p}{2})+1}}{R_k^3}\lambda^{p-2-2q}|\Omega_R|\max\left\{1, \frac{R_k}{R^\alpha}\right\}.$$ By Lemma \ref{lem:upgrader}, $$\lambda^{p-2-2q} \lessapprox R_k^{\frac{p}{2}-1-q}R^{\beta\left(\frac{p}{2}-1-q\right)}=\frac{1}{R^2}R_k^{\frac{p}{2}-1}R^{\beta\left(\frac{p}{2}-1-q\right)}.$$ Thus, \begin{align*}
    \lambda^p|U_\lambda\cap \Lambda_k|&\lessapprox R^{\frac{p}{2}\beta}\frac{R_k^{\frac{p}{2}-1}}{R_k^3R}|\Omega_R|\max\left\{1, \frac{R_k}{R^\alpha}\right\}.
\end{align*} Note $$\frac{R_k^{\frac{p}{2}-4}}{R}\le \frac{(R^\beta)^{\frac{\alpha+1}{\beta}-1}}{R}= \frac{R^\alpha}{R^\beta},$$ and so we get $\lambda^p|U_\lambda\cap \Lambda_k|\lessapprox R^{\frac{p}{2}\beta}|\Omega_R|.$

\section{Broad-to-linear Reduction - Proof of Theorem \ref{thm:mainResult}}\label{sec:multRed}
Up to this point, we have shown Theorem \ref{thm:p<11}. Let us restate: for any $\tau_1, \ldots, \tau_4$ caps with $|\tau_i|\sim 1$ and $\sim 1$-separated, we have \begin{equation}\label{eq: MultMainResult}
    \|\min_{1\le i\le 4}|F_{\tau_i}|\|_{L^{p}(\Omega_R)}^{p}\lessapprox R^{\beta \frac{p}{2}} |\Omega_R|\max_{\gamma}\|F_\gamma\|_\infty^{p}.
\end{equation}
  To show the linear inequality for $p\ge 11$, we follow an argument similar to \cite[Section 2]{maldague2024smallcapdecouplingmoment}. Let $C(R)$ be the best constant satisfying $$\|F\|_{L^{p}(\Omega_R)}^{p}\le C(R)R^{\beta \frac{p}{2}}|\Omega_R|\max_{\gamma}\|F_\gamma\|_{\infty}^{p}.$$ By iterating the following inequality, we have the desired conclusion $C(R)\lessapprox 1$.

\begin{lemma}\label{lem:BNinduction}
    For any $\epsilon>0$, $$C(R)\lesssim_\epsilon C(R^{1-\epsilon})+R^{1000\epsilon}.$$
\end{lemma}

\begin{proof}
    Fix $\epsilon>0$ and let $K\sim R^{\epsilon\beta}$. By a standard broad-narrow argument, we have $$\|F\|_{L^{p}(\Omega_R)}^{p}\lesssim \sum_{|\tau|=K^{-1}}\|F_\tau\|_{L^{p}(\Omega_R)}^{p}+K^{100}\sum_{\substack{|\tau_i|=K^{-1}\\ d(\tau_i, \tau_j)\ge \frac{1}{10}K^{-1}}}\|\min_{1\le i\le 4}|F_{\tau_i}|\|_{L^{p}(\Omega_R)}^{p},$$ where the first sum is over all caps $\tau$ with length $\frac{1}{K}$ and the second sum is over all quadruples of caps $(\tau_1, \tau_2, \tau_3, \tau_4)$ with length $\frac{1}{K}$ and $\frac{1}{10K}$ separated.  The narrow term we compute by rescaling and the induction hypothesis. Indeed, fix $\tau$ and let $T$ be an affine transformation rescaling $\tau$ to the full moment curve over $[\frac{1}{2}, 1]$ and $\wh{G}(\xi) = \wh{F}(T^{-1}\xi)$. Note $|\det T|\sim K^{10}$. This gives $$\|F_{\tau}\|_{L^{p}(\Omega_R)}^{p} \lesssim \frac{1}{K^{10(p-1)}} \|G\|_{L^{p}(T^{-1}(\Omega_R))}^{p}.$$ Note that $T^{-1}(\Omega_R)$ is roughly a $$\frac{R}{K}\times \frac{R}{K^2}\times \frac{R}{K^3}\times \frac{R^{\alpha}}{K^{4}}$$ regular box. We will trivially extend this to a $$\frac{R}{K}\times \frac{R}{K^2}\times \frac{R}{K^{\frac{1}{\beta}}}\times \frac{R^{\alpha}}{K^{\frac{\alpha}{\beta}}}$$ regular box. Applying the induction, we get \begin{align*}
        \|G\|_{L^{p}(T^{-1}(\Omega_R))}^{p}&\lesssim C(\frac{R}{K^{\frac{1}{\beta}}})\left(\frac{R}{K^{\frac{1}{\beta}}}\right)^{\beta \frac{p}{2}}\frac{|\Omega_R|}{K^{3+\frac{\alpha+1}{\beta}}}\max_{|\gamma^\prime|=KR^{-\beta}}\|G_{\gamma^\prime}\|_{\infty}^{p}\\
        &=  C(\frac{R}{K^{\frac{1}{\beta}}})\frac{R^{\beta\frac{p}{2}}}{K^p}|\Omega_R|\max_{|\gamma^\prime|=KR^{-\beta}}\|G_{\gamma^\prime}\|_{\infty}^{p}.
    \end{align*} By reversing the change of variables, we get $$\max_{|\gamma^\prime|=KR^{-\beta}}\|G_{\gamma^\prime}\|_{\infty}^{p}\lesssim K^{10p}\max_{\gamma}\|F_\gamma\|_{\infty}^{p}.$$
    We combine the previous inequalities to get $$\sum_{|\tau|=K^{-1}}\|F_{\tau}\|_{L^{p}(\Omega_R)}^{p}\lesssim 
    \frac{1}{K^{p-11}}C(R/K^{1/\beta})R^{\beta \frac{p}{2}}|\Omega_R|\max_{\gamma}\|F_\gamma\|_{\infty}^{p}.$$ Since $p\ge 11$, we get $$\frac{1}{K^{p-11}}\le 1,$$ and so, we get the correct bound for the narrow term. By a similar rescaling argument and using \eqref{eq: MultMainResult}, we also get $$K^{100}\sum_{\substack{|\tau_i|=K^{-1}\\ d(\tau_i, \tau_j)\ge \frac{1}{10}K^{-1}}}\|\min_{1\le i\le 4}|F_{\tau_i}|\|_{L^{p}(\Omega_R)}^{p}\lessapprox K^{O(1)}R^{\beta\frac{p}{2}}|\Omega_R|\max_{\gamma}\|F_{\gamma}\|_{\infty}^{p}.$$
\end{proof}

\section{Proof of Corollary \ref{cor:NewExpSums}}\label{sec:FinalCor}
Let $B= [0,N^a]^3\times [0,N^b]$ and $B_0= [0,N^a]^4$. By periodicity, it suffices to show $$\int_B \left|\sum_{n\sim N}c_ne(x\cdot \Phi\left(\frac{n}{N}\right))\right|^{p} dx \lessapprox N^{\frac{p}{2}}|B|.$$
We aim to apply Theorem \ref{thm:mainResult}. Let $R=N^a$, $\beta = \frac{1}{a}$, and $\alpha = \frac{b}{a}$. Note $\beta \in [\frac{1}{3}, \frac{1}{2}]$ and $\alpha \le \beta$. Moreover, let $f$ be a Schwartz function such that $f(x) \gtrsim 1_{B_0}(x)$ and $\wh{f}$ is supported on $[0, \frac{1}{100 R}]^4$. Let $$F(x)= \sum_{n\sim N}c_ne(x\cdot \Phi\left(\frac{n}{N}\right))f(x).$$ We get the desired Fourier support: $$\wh{F}(\xi) = \sum_{n\sim N} c_n \wh{f}(\xi - \Phi\left(\frac{n}{N}\right)).$$ Therefore, by Theorem \ref{thm:mainResult}, $$\int_B \left|\sum_{n\sim N}c_ne(x\cdot \Phi\left(\frac{n}{N}\right))\right|^{p} dx\lesssim \int_B |F|^p \lessapprox R^{\beta \frac{p}{2}}|B|= N^{\frac{p}{2}}|B|.$$

\bibliographystyle{plain}
\bibliography{ref}
\end{document}